\documentclass[12pt]{amsart}
\usepackage{graphicx, color, comment}
\newtheorem{thm}{Theorem}[section]
\newtheorem{cor}[thm]{Corollary}

\newtheorem{prop}[thm]{Proposition}
\theoremstyle{definition}
\newtheorem{defn}[thm]{Definition}
\theoremstyle{remark}
\newtheorem{rem}[thm]{Remark}
\numberwithin{equation}{section}

\def\Cr{color{red}}
\def\Cb{color{blue}}

\def\tr{\mathrm tr}
\def\Tr{\mathrm Tr}

\def\a{\alpha}

\def\b{\beta}
\def\g{\gamma}

\def\l{\lambda}

\def\CC{\mathbb C}
\def\RR{\mathbb R}

\def\HH{\mathbb H}
\def\Tr{{\rm Trace} \;}
\def\RR{\mathbb R}

\def\Cr{\color{red} }
\def\Cb{\color{blue}}

\def\Ll{l}

\def\[[{\bf [}
\def\]]{{\bf ]}}

\begin{document}

\title[]{Primitive Curve Lengths on Pairs of Pants.}%
\author{Jane Gilman}
\address{Mathematics Department, Rutgers University,Newark, NJ 07102}%
\email{gilman@rutgers.edu}%


\date{submitted 5/25/2015; revised 5/3/2016}
\dedicatory{Dedicated to Gerhard Rosenberger and Dennis Spellman in their 70th birthdays}
\begin{abstract}
    The problem of determining whether or not a non-elementary subgroup of $PSL(2,\CC)$ is discrete is  a long standing one. The importance of two generator subgroups comes from J{\o}rgensen's inequality which has as a corollary the fact that a non-elementary subgroup of $PSL(2,\CC)$ is discrete if and only if every non-elementary two generator subgroup is.
A solution even in the two-generator $PSL(2,\RR)$ case appears to require an algorithm that relies on a the concept of {\sl trace minimizing} that was initiated by Rosenberger and Purzitsky in the 1970's  Their work has lead to many discreteness results and algorithms. Here we show how their concept of trace minimizing leads to a theorem that gives bounds on the hyperbolic lengths of curves on the quotient surface that are the images of primitive generators in the case where the group is discrete and the quotient is a pair of pants. The result follows as a consequence of the Non-Euclidean Euclidean algorithm.
\end{abstract}
\maketitle
\section{Introduction}
The problem of determining whether or not a non-elementary subgroup of $PSL(2,\CC)$ is discrete is  a long standing one. The importance of two generator subgroups comes from J{\o}rgensen's inequality which has as a corollary the fact that a non-elementary subgroup of $PSL(2,\CC)$ is discrete if and only if every non-elementary two generator subgroup is. A solution even in the $PSL(2,\RR)$ case appears to require an algorithm that relies on a the concept of {\sl trace minimizing}. This concept was initiated and its use pioneered by Rosenberger and Purzitsky in the 1970's \cite{P,PR,R, Ros} Their work has lead to many discreteness results and algorithms. For a more recent summary see \cite{FineRo}. Here we show how their concept of trace minimizing leads to a theorem that gives bounds on the hyperbolic lengths of primitive curves on the quotient surface in the case where the group is discrete and the quotient is a pair of pants. A primitive curve is, of course, the image of a primitive generator in a rank two free group.

The main result of this paper appear in section {\ref{section:results} (Theorems {\ref{thm:stoppinglength},\ref{thm:alternate} \ref{thm:intersection} \ref{thm:seam}).
 When the group is discrete and free, the conjugacy class of any  primitive curve
corresponds to a rational number. The number of so-called {\sl essential self intersections} of the image of the curve on the quotient is given by a formula that depends upon the rational \cite{GKwords}. In particular,  if $r=p/q$ is the rational in lowest terms, let $I(r) =I(p/q)$ denote the number of essential self-intersections.   The discreteness algorithm finds the length of the three shortest simple closed curve on the pair of pants,  as well as the longest {\sl seam length}  on any quotient. The length of a curve on the quotient is given by the translation length of the matrix and is found using the trace of the matrix. We are able to bound the length of any primitive curve in terms of $I(p/q)$ and the lengths of the three shortest geodesics, Theorem
\ref{thm:intersection}. We obtain other variants on the upper and lower bounds for the translation length of any primitive element of the group using the entries in the continued expansion of $r$ and the translation length of the shortest curves,  Theorem \ref{thm:stoppinglength},  and we express the longest seam length as a limit, Theorem \ref{thm:seam}.

\vskip .2in
The results stem from casting the discreteness algorithm as a Non-Euclidean Euclidean algorithm \cite{NEE}.
\vskip .2in

We note that other authors bound the lengths of curves, primitive or not, in terms of the number of self intersections (see \cite{BAI,Chas1,CL} and references given there).




The paper builds on a theory of curves on surfaces and algorithms, especially the Non-Euclidean Euclidean Algorithm (the NEE for short)  and is organized with a number of preliminary sections which can be viewed as an exposition for the uninitiated or skipped for those familiar with the theory.  The new results appear in section \ref{section:results}. More specifically in section \ref{section:not} notation is set, in section \ref{section:prior} the Gilman-Maskit algorithm is given (the GM algorithm for short), the $F$-sequence is defined and the Non-Euclidean Euclidean Algorithm is given. Section \ref{section:hexagons} reviews the standard notation for the hyperbolic hexagon determined by a two-generator group and in   section \ref{section:wind} the winding and unwinding process of the algorithm is explained.



\section{Preliminaries: Notation and terminology } \label{section:not}

In this section we set notation and review terminology. In the next sections we summarize results that will be needed.

In what follows we let $G= \langle A,B \rangle$
be a non-elementary subgroup of $PSL(2,\CC)$ or equivalently the group of isometries of hyperbolic three-space $\mathbb{H}^3$.

We will concentrate on $PSL(2,\RR)$ which we identify with the group of isometries of hyperbolic two-space. We recall that a matrix $M =\left(
                                         \begin{array}{cc}
                                           a& b \\
                                           c & d \\
                                         \end{array}
                                       \right)$ with $ad-bc=1$ in $SL(2,\RR)$ acts on the upper-half plane $\HH^2 = \{ z = x +iy \;| \; x, y \in \RR,  y >0\}$ by $z \rightarrow{\frac{az+b}{cz+d}}$.
                    That is, a matrix induces a fractional linear transformation. Since a matrix and its negative yield the same action,  we identify
the action with that of $PSL(2,\RR)$. Further if $\HH^2$ is endowed with the hyperbolic metric, these transformations are isometries in the metric.
We let $\tr (M)$ denote the trace of $M$.

We recall that elements of the isometry group $\HH^2$ are classified  by their traces and equivalently by their geometric
actions. If $|\tr(M)|  > 2$, $M$ is hyperbolic and fixes two points on the boundary of $\HH^2$ (i.e. the real axis) and a hyperbolic or non-Euclidean straight line. This is the semi-circle in the upper-half plane perpendicular to the real axis intersecting it at the fixed points of $M$. The transformation moves point on the semi-circle towards one fixed point, the attracting fixed point, and away from the other fixed point, the repelling fixed point and moves points a certain fixed distance in the hyperbolic metric. This distance is known as the translation length of $M$ and we denote it by $T_M$. The semi-circle is termed the axis of $M$ and denoted $Ax_M$. The formula relating the translation length to the trace of the matrix is given below.

However, first we note that if $A$ and $B$ are any generators, there are two possible pull-backs from $PSL(2,\RR)$ to $SL(2,\RR)$, one with positive trace and one with negative trace. If we choose pull-backs for both $A$ and $B$ with positive traces, then the trace of any word in the pull-back generators is determined and for simplicity of exposition we denote the trace of an element $X \in G$ by $\Tr X$. Note that the trace of any commutator is always positive irrespective of pull backs.

For ease of exposition we do not distinguish rotationally between
an element of $X \in G = \langle A,B \rangle \subset PSL(2,\RR)$ and the corresponding matrix or isometry with the trace $\Tr X$ as defined above.

  An element $X$ in $SL(2,\RR)$  or its inverse is conjugate to a fractional linear transformation of the form: $z \mapsto \pm Kz$ for some real number $K > 0, K \ne  1$.   $K$ is called the multiplier of $X$ and is denoted  $K_X$ and $|\Tr  X| = \sqrt{K} + {\sqrt{K}}^{-1}$.

\vskip .2in

$T_X$ and $K_X$ are related by
$$ \cosh {\frac{T_X}{2}} = {\frac{1}{2}}|\Tr X| = {\frac{1}{2}}(\sqrt{K} + {\frac{1}{ {\sqrt{K}} }} )$$
where $\cosh$ is the hyperbolic cosine.
\vskip .2in We note for future use,
 $$ T_X = T_{X^{-1}} = |\log \; K_X| $$
Note that if $|tr (M)| = 2$, $M$ is parabolic and fixes one point on the real axis, known as its axis. It does have a rotation direction about its fixed point even though it does not have a translations length. If $|tr(M)|  < 2$, $M$ is elliptic and is a rotation about its single fixed point in upper-half plane.

If $A$ and $B$ are any hyperbolic elements, it is well known that their axes have a common perpendicular $L$. Orient the axis towards their attractive fixed points. We say that the pair $(A,B)$ is {\sl coherently oriented} if when $L$ is oriented from the axis of $A$ to the axis of $B$ the attracting fixed points of $A$ and $B$ lie to the left of $L$ and  $\Tr A \ge \Tr B \ge 0$.

If one or both of $A$ and/or $B$ is parabolic, the is still a common perpendicular $L$,  either between the axis of the hyperbolic and the axis of the parabolic, that is its fixed point of the parabolic or between the two parabolic fixed points, and the definition of coherent orientation still applies provided if when only one generator is hyperbolic it is $A$.


If $X$ is any geodesic in $\HH^2$ we let $H_X$ be the half-turn about $X$.
\vskip .2in

Now the $PSL(2,\RR)$ discreteness algorithm divides naturally into two cases, the intertwining case known as the GM algorithm \cite{GM} and the intersecting axes algorithm \cite{JGtwo}. Here we only address the intertwining algorithm, the case where any pair of hyperbolics encountered have disjoint axes.
\section{Preliminaries: Summary of Some Prior Results}
\label{section:prior}

\begin{thm} \cite{GM}  {\rm The Gilman-Maskit Geometric $PSL(2,\RR)$ Algorithm   (1991)}
\vskip .1in
Let $G= \langle A, B \rangle$ be non-elementary with $A,B \in PSL(2,\RR)$.
Assume that $A$ and $B$ are a  coherently  ordered pair of generators and that  $A$ and $B$ have disjoint axes in the case when both are hyperbolic. There exists integers $[n_1,n_2,...,n_t]$ that determines
the sequence of pairs of generators
$$(A,B) \rightarrow (B^{-1}, A^{-1}B^{n_1}) \rightarrow (B^{-n_1}A, B(A^{-1}B^{n_1})^{n_2}) \rightarrow \cdots \rightarrow (C,D)$$
which is a sequence of coherently ordered pairs and which stops at a pair $(C,D)$ after a finite number of replacements and says
(i) G is discrete
(ii) G is not discrete, or
(iii) G is not free.
\end{thm}
\vskip .1in
We emphasize that the replace is done in a {\sl trace minimizing manner}.
That is, if we set $(A,B) = (A_0,B_0)$ and at step $j \ge 1$ consider the ordered $(A_j,B_j)$ pair. We can then assume by the algorithm assume
$$|\Tr \;
A_j| \ge |\Tr \;  B_j| \; \mbox{ and }$$
and at the next pair $(A_{j+1},B_{j+1})$,  which is equal to the ordered pair $(B_j^{-1}, A_j^{-1}B_j^{n_{j+1}})$, we have
$$|\Tr \; B_j| \ge |\Tr \;  B_j^{-n_{j+1}}A_j |$$

The algorithm stops when it reaches an element $X$ with  $\Tr X \le 2$.
\vskip .1in
We refer to this geometric discreteness algorithm as the GM algorithm.
\vskip .2in
The sequence of integers has been used in  a number of settings. For example it was used to calculate  the computational complexity of the algorithm in (1997) by Gilman \cite{JGalg} and by Y. C. Jiang  in his (2001) thesis \cite{YCJ} and to calculate the essential self-intersection numbers of primitive curves on the quotient when the group is discrete \cite{GKwords}.
However, it was not until 2002 in \cite{GKwords} after the sequence has been used in various settings, that it was given a name.
\begin{defn} \cite{GKwords} (2002) The sequence $[n_1,...,n_t]$ is termed the $F$-sequence or the {\sl Fibonacci sequence} of the algorithm.
\end{defn}
The existence of such $n_i$ was implicit in the proof of the GM geometric discreteness algorithm.
However, these numbers were never actually computed.

The interpretation of the  algorithm as a non-Euclidean Euclidean Algorithm  gave a theorem that showed   how to actually compute the $n_i$ \cite{NEE}.

We note that if the algorithm begins with a pair of hyperbolics, the new pair will either be a pair of hyperbolics or a hyperbolic-parabolic pair or it will stop. If it arrives at or  begins with a hyperbolic-pair, the next pair with either be a hyperbolic-parabolic pair again  or stop or be a parabolic-parabolic pair. A parabolic-parabolic may repeat or be a stopping pair. It repeats a given type of pair at most a finite number of times and then moves on the next type or stops.
The algorithm can be extended to apply to elliptics of finite order \cite{JGalg} but we do not consider that case here as the definition of the terms in the $F$-sequence needs to be modified (see \cite{Vidur}).

\begin{rem} \label{rem:order}
The algorithm stops and says that $G$ is discrete at a pair $(C_t,D_t)$ when $\Tr C_t \ge \Tr D_t
\ge 2$ and $\Tr C_t^{-1}D_j \le 0$. If one orders the stopping generators in terms of their traces or the length of the corresponding curves on the quotient surface there are three possibilities
\begin{enumerate}

\item $\Tr C_t \ge \Tr D_t \ge |\Tr C_t^{-1}D_t|$.
\item $\Tr C_t \ge  |\Tr C_t^{-1}D_t| \ge \Tr D_t
$.
\item $ |\Tr C_t^{-1}D_t| \ge \Tr C_t \ge \Tr D_t$.

\end{enumerate}

When the algorithm stops and says that the group is discrete, there is a right angled hexagon with disjoint sides (see section \ref{section:wind} for more details). For each parabolic generator, one of the six sides reduces to a point. In such a case $|\Tr C_t^{-1}D_t|=2$ and one or two or none of the other traces may have absolute value $2$.

\end{rem}
\begin{thm} \label{thm:NEE} \cite{NEE} {\rm {\bf The Non-Euclidean Euclidean Algorithm}   (2013)}
 Let $G  = \langle A, B \rangle$ where $A$ and $B$ are a coherently ordered pair of transformations in $PSL(2,\RR)$ whose axes are disjoint in the case of a pair of hyperbolics.

      If one applies the Euclidean algorithm to the non-Euclidean translation lengths of the of  generators at each
step and the algorithm stops at a pair of generators saying that the group is discrete, the output of the Non-Euclidean Algorithm is the F-sequence $[n_1,...,n_k]$. It is and is obtained by applying the Euclidean algorithm to the non-Euclidean translation lengths of the of  generators at each step to the possible cases as follows:

\vskip .2in
\noindent {\bf  {\sl A. If the initial $A$ and $B$ are hyperbolic generators with disjoint axes
 and the algorithm stops at a pair of hyperbolics saying that the group is discrete}},
\vskip .2in
one has

$$ n_1= \mbox{{\LARGE{\bf [}}} {\frac{(|\log K_A|)/2} {(|\log K_B|)/2} } \mbox{{\LARGE{\bf ]}}} =  \mbox{{\LARGE{\bf [}}} {\frac{T_A/2}{T_B/2}} \mbox{{\LARGE{\bf ]}}}$$

$$\mbox{where} \;  \mbox{{\LARGE{\bf [}}} {\;\;\;} \mbox{{\LARGE{\bf ]}}} \mbox{denotes the greatest integer function.}$$


$$n_2 = \mbox{{\LARGE{\bf [}}} {\frac{T_B/2}{T_D/2}}\mbox{{\LARGE{\bf ]}}} \;\;\; \mbox{where} \;\;\; D = A^{-1}B^{n_1}$$
and
$$ n_j = \mbox{{\LARGE{\bf [}}} {\frac{T_{C_j}/2}{T_{D_j}/2}}\mbox{{\LARGE{\bf ]}}}$$

Here  $(A,B)=(C_0,D_0)$ and $(C_j,D_j) = (D_{j-1}^{-1}, C_{j-1}^{-1}D_{j-1}^{n_j})$
is the ordered pair of generators at step $j$, $0 \le j \le t$.
\vskip .2in

\noindent {\bf {\sl B.  If one begins at any type of pair and stop at a pair that includes one or more parabolics,}}
\vskip .2in
one has
\vskip .2in
(i) {\sl If at step $j$,  $(C_j,D_j)$
 is  hyperbolic-parabolic pair,}  then
 $$n_j = \mbox{{\LARGE{\bf [}}}   { \frac{{\Tr (C_j)}-2}{ {\sqrt{|{\Tr ([C_j,D_j])-2}|} } }}\mbox{{\LARGE{\bf ]}}}$$ or equivalently, setting $A=C_j$ and $B = D_j$.
 $$n_j = \mbox{{\LARGE{\bf [}}}  {\frac{2 \cosh({\frac{T_A}{2}})-2}{\sqrt{2\cosh{( {\frac{T_{[A,B]}}{2}})-2} }}} \mbox{{\LARGE{\bf ]}}}.$$

(ii)  {\sl If  at step $j$,  $(C_j,D_j)$ is a parabolic-parabolic pair}, $j = t$ and  $n_t =1$.
\vskip .05in
step(ii) {\rm continued}  Further if $t \ge 2$, at step $(t-1)$,  $(C_{t-1},D_{t-1})$ is  a hyperbolic-parabolic pair with     $$n_{t-1}= \mbox{{\LARGE{\bf [}}}   { \frac{{\Tr (C_{t-1})}-2}{ {\sqrt{|{\Tr ([C_{t-1},D_{t-1}])-2}|} } }} \mbox{{\LARGE{\bf ]}}}$$
and$$ n_t=1.$$

(iii) {\sl If the initial pair is a hyperbolic-parabolic pair, then the $F$-sequence is of length $2$} and is $[n_1,n_2]$
where
$$n_1 = \mbox{{\LARGE{\bf [}}}  {\frac{2 {\cosh({\frac{T_A}{2}})}-2}{\sqrt{2\cosh{( {\frac{T_{[A,B]}}{2}})-2}} }}\mbox{{\LARGE{\bf ]}}}$$
 and $$n_2=1.$$
(iv) {\sl If the initial pair is a parabolic-parabolic pair,} the $F$-sequence is of length $1$ with $n_1=1$.
\end{thm}


\section{Preliminaries: Standard Hexagons, Notation and Facts} \label{section:hexagons}
We recall that to any group two generator group $G= \langle A, B \rangle$ one associates a standard three generator group $3G = \langle H_L, H_{L_A}, H_{L_B} \rangle$ where $L$ is the common perpendicular to the axes of $A$ and $B$  and $L_A$ and $L_B$ are geodesics such
 $A = H_L H_{L_A}$ and $B = H_L H_{L_B}$.
The six geodesics $Ax_A, L_A, Ax_B,L_B, A_{A^{-1}B}, L$ determine a right angled hexagon in $\HH^2$. The hexagon will have self-intersecting sides
with the exception of the hexagon determined by the stopping generators when the groups is discrete. This hexagon will be convex. In this case the projections of the axes sides will be the three shortest geodesics on the quotient \cite{GKwords}. We speak of the interior sides of the hexagon which are the portions of the axes of the generators between two half-turn axes and the portions of the half-turn axes between two axes. It will be clear from the context whether we are speaking about a geodesic or the portion of the geodesic interior to the hexagon.

When $G$ is discrete, we let $\pi$ be the projection of $\HH^2$ onto $\HH^2/G$. The quotient is a three-holed sphere or a sphere with up to three punctures.



If $G$ is generated by $W$ and $V$, then $W$ and $V$ are called {\sl primitive associates} and $(W,V)$ is called a {\sl primitive pair}.

 The primitive elements project to closed curves on the quotient which we call the {\sl primitive curves}
and we may assume primitive curves are geodesics, as are all curves mentioned.
We let the generators that stop the algorithm be $\tilde{A}_0,\tilde{B}_0,\tilde{C}_0=
{\tilde{A}}_0^{-1}{\tilde{B}}_0$ and
 $\a_0,\b_0,\g_0$ their respective images on the surface.
 These are the three shortest curves on the surface \cite{GKwords} and ${\Ll}_0, l_{\a_0}, \mbox{ and } l_{\b_0}$, the corresponding seams are the three longest seams.

We denote the length of the images of the stopping generators  by $L(\a_0), L(\b_0),\mbox{and} L(\g_0)$. The trace convention allows us to assume that  $L(\a_0) \ge  L(\b_0)$ We allow the possibility of parabolic stopping generators whose lengths are zero. In that case, we set $L({\hat{\a}}_0)$ to be the length of the first non-zero winding curve,
that is $L(\pi({\tilde{A}}_0{\tilde{B}}_0))$.

\section{Preliminaries: Essential Self-intersections and winding and unwinding sequences}\label{section:wind}

Any primitive  geodesic on $\HH^2/G$  will have a certain number of self-intersections on the images of the  half-turn lines. There are termed {\sl essential self-intersections}.  All other self-intersections are self-intersections that will be undone  when an essential self intersection is cut. These are  termed {\sl trailing intersections}.  If a primitive curve that is not a stopping generator has $F$-sequence $[n_1,...n_t]$, then there a formula for essential self intersections. This is theorem 6.1 of \cite{GKwords}. Note that the formula in Theorem 7.1 part (7) has a typo but the formula in Theorem 6.1 is correct.

Before we can present the formula, we need to consider {\sl winding}  and {\sl unwinding} sequences.
Since the algorithm ends with curves that are simple, we view the algorithm as an unwinding procedure \cite{GKwords}. We can run the algorithm backwards as a winding procedure.

\begin{thm} \cite{GKwords} Proposition 1.6.
If $F$-sequence or the unwinding sequence is $[n_1,...,n_t]$, then  the winding sequence is
$[-n_t, ...,-n_1]$.
\end{thm}

However, we can also write the winding sequence with positive integers, adjusting the orientation of curves carefully as in \cite{GKwords} so that the winding sequence is $[n_1,...,n_t]$ and the corresponding continued fraction entries are $[n_0,n_1, ...,n_k]$ where $n_0=0$.






\begin{thm} \cite{GKwords} If $W$ is a primitive element in $G$, replacing $W$ by its inverse if necessary, we may assume that there is a primitive associate $V$ such that the ordered pair $(W,V)$ is coherently  ordered. After unwinding to stopping generators and then winding, $W$ corresponds to a unique positive rational $r$ given by $p/q$ where $p$ and $q$ are relatively prime positive integers.
\end{thm}

In what follows, we let $\gamma_r$ denote the  primitive curve on the quotient surface given by the rational $r=p/q$ that comes from winding a about a pair of stopping generators. Note that curve $\gamma_r$ depends only on the conjugacy class of its pre-image.
\begin{thm} \cite{GKwords}
We let $I(n_1,...,n_t)$ denote the number of essential self-intersections of the curve with unwinding sequence $[-n_t, ...., -n_1]$ so that $\gamma_r$ has winding sequence $[n_1,...,n_t]$ and $r$ is the rational continued fraction expansion   $[n_0,...,n_t]$ and we let $r_k =[n_0,...,n_k]$ denotes its k-th approximant. We set $r_k = p_k/q_k$ where $p_k$ is the numerator and $q_k$ the denominator of the approximant and $r_k$ is given in lowest terms. Then the essential self-intersection numbers are given inductively as follows:


$$I(\a_0) =0, I(\b_0) = 0, I(\a_0^{-1}\b_0) =0, I(\a_0\b_0) = 1, I(\a_0\b_0^2) =2 $$

\vskip .2in
 and $I_{p/q}$ is determined inductively by the formulas

$$I_{p_{k+1}/q_{k+1}} = 1 + n_{k+1} \cdot I_{p_{k}/q_{k}}  + I_{p_{k-1}/q_{k-1}}$$

\end{thm}

We write $I(r)$ and $I(p/q)$ for the number of essential self-intersections.


\section{Length Inequalities} \label{section:results}


In this section we assume that group is discrete and free and that the stopping generators are the coherently oriented pair $({\tilde{A}}_0,\tilde{B}_0)$. The corresponding hexagon is right angled and convex and we let $L(\l_0)$ be the largest seam length of the corresponding pair of pants.
As before let $\a_0, \b_0, \g_0$ denote the images of the sides of the hexagon on the quotient and $L(\delta)$ the length of any geodesic, $\delta$,  on the quotient surface.


By \cite{GKwords} and the formula for $\cosh \pi(L)$ given below and with this orientation the length of the longest seam is the largest length between any associate primitive pair of curves on the surface.

Let $W \in G$ be primitive and assume that after replacing $W$ by its inverse if necessary there is a primitive $V$ such that $(W,V)$ is a coherently oriented primitive pair. There is the unwinding sequence given by $(C_j,D_j)$ where $j=0,...,t$ and $W=C_0$ and $V=D_0$ which stops at generators $\{ \tilde{A}_0, \tilde{B}_0, {\tilde{A}}_0^{-1}\tilde{B}_0 \}$ where the stopping generators are $C_t, D_t$
have been renamed taking into account any necessary normalization needed following remark \ref{rem:order}. That is,  we now assume that $L(\a_0) \ge L(\b_0) \ge L(\g_0)$.
 and that the corresponding winding sequence begins at $\tilde{A}_0, \tilde{B}_0, {\tilde{A}}_0^{-1}{\tilde{B}}_0$ is given by $[m_1, ..., m_t]$ and the rational $r$ has the  continued fraction entries $[0,m_1,...,m_t]$.

We note that if $(W,V)$
is a primitive pair, it determines a set of stopping generators. Any other set of stopping generators will be conjugate to this set. 

In the winding case where we start with three coherently oriented parabolics, we use part {\bf B} of theorem \ref{thm:NEE} backwards.

\begin{thm} \label{thm:stoppinglength} {\rm {\bf Winding Sequences and Translation Lengths}}

Let $W$ be any primitive element of $G = \langle A, B \rangle$ where $G$ is discrete and free and has stopping generators $(\tilde{A}_0, \tilde{B}_0)$. Assume that $[n_1,...,n_t]$ is the $F$-sequence of $W$. Assume $W$ is not a stopping generator.

Assume that not all three of the stopping generators are parabolic. With the normalization above where
${\frac{T_{{\tilde{A}}_0}}{2}}
= {\frac{L(\a_0)}{2}}$
is half the translation length of the longest hyperbolic stopping generator,  we have $${\Big (} \Pi_{i=1}^t n_i {\Big )} \cdot  {\frac{T_{{\tilde{A}}_0}}{2}} < {\frac{T_W}{2}} \le {\Big (} \Pi_{i=1}^t (n_i +1) {\Big )} \cdot {\frac{T_{{\tilde{A}}_0}}{2}}.$$

If all three stopping generators are parabolic, set $\hat{A}_0= \tilde{A}_0^{-1}{\tilde{B}}_0$. We have
  $${\Big (}\Pi_{i=1}^t n_i {\Big )} \cdot  {\frac{T_{{\hat{A}}_0}}{2}} < {\frac{T_W}{2}} \le {\Big (} \Pi_{i=1}^t (n_i +1) {\Big  )}\cdot {\frac{T_{{\hat{A}}_0}}{2}}.$$

Equivalently, since $|m_k| = |n_{t-k}|$, the formulas can be written with $m_i$ replacing $n_i$.
\end{thm}
\vskip .1in
\begin{proof} We first assume that the stopping generators are all hyperbolics and that ${\tilde{A}}_0$ is the generator with largest translation length. We note for the unwinding sequence the Non-Euclidean Euclidean Algorithm tells if we remove the greatest integer symbols, for each for each integer $j$ we have
$$n_j{\frac{T_{C_{j+1}}}{2}} \le {\frac{T_{C_j}}{2}} \le (n_j+1){\frac{T_{C_{j+1}}}{2}}.$$
where the unwinding sequence is $[n_1,...n_t]$ and the winding sequence taken to be positive is $[m_1,...,m_t]$ where $m_k= n_{t-k}$. This gives the theorem. If the algorithm includes one or more parabolic elements before it stops, we have the same formula except we must re-interpret the initial few $m_i$'s.  That is, in  the winding case where we start with three coherently oriented parabolics, we use part {\bf B} of theorem \ref{thm:NEE} backwards.\end{proof}

The theorem can be rephrased as

\begin{thm} \label{thm:alternate} {\rm {\bf Winding Sequences and Curve Lengths}}

Let $W$ be any non-parabolic primitive element of $G = \langle A, B \rangle$ where $G$ is discrete and free and has stopping generators $(\tilde{A}_0, \tilde{B}_0)$. Assume that $W$ is not conjugate to a stopping generator. Let $W$ have $F$-sequence $[n_1,...,n_t]$ and let $L(W)$ be the length of its image on the surface.   Let $L(S_0)$ be longest length of any simple curve on the quotient surface.
$${\Big (} \Pi_{i=1}^t n_i {\Big )} \cdot  L(S_0) < L(W)  \le {\Big (} \Pi_{i=1}^t (n_i +1) {\Big)} \cdot L(S_0).$$ Equivalently, since $|m_k| = |n_{t-k}|$, the formulas can be written with $m_i$ replacing $n_i$.
\end{thm}
\vskip .2in
We recall from \cite{GM} that if we begin with a coherently ordered pair of hyperbolics, then as long as we stay in the hyperbolic-hyperbolic case with all traces positive, at each small step going from replacing  $X$ and $Y$ by $XY^{-1}$ and $Y$, the trace is reduced by at least
${\frac{(\sqrt{2}-1)^2}{{\sqrt{2}}} }$.

\begin{thm} \cite{JGalg} {\bf {\rm Counting Hyperbolic Repetitions}}\label{thm:reduced}
Assume that $G$ is discrete and free and that we unwind  beginning  with $(A,B)$ and ending  with $(A_0,B_0)$,  all  hyperbolics with disjoint axes. Let $[n_1,...,n_t]$ be the unwinding sequence. Then
$${\frac{(\sqrt{2}-1)^2}{{\sqrt{2}}} }{\Big (}  \Sigma_{1=1}^t n_t {\Big )} \ge  \Tr \; A -2$$
and the algorithm repeats the hyperbolic-hyperbolic case at most $q$ times where $q = {\Big (}  \Sigma_{1=1}^t n_t {\Big )} $ \end{thm}

In \cite{JGalg} similar bounds for $t$ were found for other cases.
\vskip .02in
\begin{thm}\label{thm:intersection}{\rm {\bf Minimal curve length and essential intersections}}
Continuing with the notation above, let $\gamma_r$ denote the primitive curve on the quotient corresponding to the rational  $r = p/q$. Then its length $L(\gamma_r)$ satisfies
$$ I(r) \times L(\g_0) \le L(\gamma_r) \le (I(r) + 1) \times L(\a_0)$$
where $\gamma_0$ is the shortest stopping generator and $\alpha_0$ the longest.\end{thm}
\vskip .2in
\begin{proof}
This follows from the fact that one plus the number of essential-self intersections correspond to winding around curves of length
of length at most $L(\a_0)$ and at least $L(\g_0)$.
Every essential self-intersection corresponds to the image of a segment that runs between either the images of $L$ and $L_A$, $L$ and $L_B$ or
$L_B$ and $L_A$. Thus the images of a segment is of length at least $L({\frac{\gamma_0}{2}})$. An essential self-intersection will correspond to traversing a segment between two seams twice.
\end{proof}

We recall from Fenchel \cite{Fench} that if a convex right angled hexagon has alternating sides  of length $x \le y \le  z$ then the length of the interior side $L$ opposite the shortest side $x$
satisfies:$$ \cosh L = {\frac{\cosh x + \cosh y \cosh z}{\sinh y\sinh x}}.$$
From this we can compute the hyperbolic length of the longest seam.

We consider the set of curves $AB^j$, $j =1, ... \infty$ in $\HH^2$. We let $q_i$ be the point on the $L$, their common perpendicular  where $AB^i$ intersects $L$. We note from \cite{NEE,Vidur} $q_{i+1}$ lies to the right of $q_i$ given a coherent orientation when we are at a pair of stopping generators. We set $q_0$ be the point where the axis of $A$ intersects $L$ and $\rho_i$ the hyperbolic distance from $q_{i-1}$ to $q_i$, $i=1,...\infty$.

We have
\begin{thm} \label{thm:seam} {\rm {\bf Seam Lengths and Intersections}}

Assume $G$ is discrete and free with stopping generators  ${\tilde{A}}_0$ and ${\tilde{B}}_0$ and that $L$ is their common perpendicular and let $L(\l_0)$ be the length of $L$.
If $\rho_i$ is the distance between the intersections of ${\tilde{A}}_0{\tilde{B}}_0^{i-1}$ and
$ {\tilde{A}}_0{\tilde{B}}_0^i$
with $L$, $i=1,... $ , then
$$\lim_{t \rightarrow \infty} \Sigma_{i=1}^t \rho_i = L(\l_0).$$
\end{thm}
\begin{proof} Assume for ease of notation that the stopping generators are $A$ and $B$.
We recall that the transformation $B$ can be factored in an infinite number of ways as the product of two half-turns about geodesic perpendicular to its axis which are half its translation length apart along the axis of $B$. We write these lines $$L_B,L_{B^2}, .... , L_{B^j}, ...$$
and their images under the half-turn about $L$ and $$L_{\overline{B}},L_{\overline{{B}}^2}, .... , L_{{\overline{B}}^j}, ...$$
We note that $AB^j = L_AL_{\overline{B}^j}$ and these curves will intersect $L$ between the axes of $A$ and $B$.
 \end{proof}

\begin{cor} Let $l_{(C,D)}$ be a half-turn line corresponding to any primitive pair $(C,D)$ and $L(l_{(
C,D)})$ the length of its image on the quotient.
Then
$$L(l_{(C,D)} \le \lim_{t \rightarrow \infty} \Sigma_{i=1}^t \rho_i = L(\l_0)$$
where $\rho_i$ is  as in Theorem \ref{thm:seam}.

\end{cor}

\begin{proof}
The image of $L$ is the longest of that of any of the half-turn lines.
\end{proof}

\vskip .02in

\begin{rem}
\label{remark:intersecting}
In the case that the group is generated by hyperbolics with intersecting axes and the commutator is not elliptic, the group will be discrete and no algorithm is needed.  In earlier papers we noted that following the methods of \cite{NEE}, one can still obtain a finite $F$-sequence which stops at a pair where lengths are shortest. One can wind and unwind. This will be addressed elsewhere as will the cases involving elliptic elements of finite order from either the intertwining case or the intersecting axes case when the group is discrete and infinite order when the group is not. For elliptic elements of finite order one has to use the extended the notion of $F$-sequence developed by Malik in \cite{Vidur}.
\end{rem}

\begin{rem}
These inequalities and bounds have obvious applications to lengths of geodesics on pairs of pants in a pants decomposition for a compact surface of genus $g
\ge 2$. The formulas can be stated in terms of the minimal curve length on a pair of pants with one self intersection instead of the minimal non-zero curve length on the pair of pants. This will be addressed in \cite{JGnext}
\end{rem}
\end{document}